\documentclass[a4paper,12pt]{article}     
\usepackage{amsmath,amscd,amssymb}        
\usepackage{latexsym}                     
\usepackage[english, german]{babel}                

\usepackage{theorem}                 
\usepackage{color}

\newtheorem{theorem}{Theorem}
\newtheorem{corollary}{Corollary}

\newtheorem{conjecture}{Conjecture}

\theorembodyfont{\rmfamily}
\newtheorem{definition}{Definition}
\newtheorem{remark}{Remark}

\newenvironment{proof}{\begin{ProofwCaption}{Proof}}{\end{ProofwCaption}}
\newenvironment{proof*}[1]{\begin{ProofwCaption}{{#1}}}{\end{ProofwCaption}}
\newenvironment{ProofwCaption}[1]%
  {\addvspace\theorempreskipamount \noindent{\it #1.}\rm}%
  {\qed \par \addvspace\theorempostskipamount}
\newcommand{\qedsymbol}{{\rm $\Box$}}
\newcommand{\qed}{\hfill\qedsymbol}

\newcommand{\CC}{{\mathbb C}}

\newcommand{\QQ}{{\mathbb Q}}

\newcommand{\RR}{{\mathbb R}}
\newcommand{\ZZ}{{\mathbb Z}}

\def\H{{\mathcal H}}

\newcommand{\calJ}{{\cal J}}

\newcommand{\ulam}{\underline{\lambda}}

\newcommand{\orb}{{\,{\rm orb}}}
\newcommand{\Conj}{{\rm Conj\,}}
\newcommand{\age}{{\rm age}}
\newcommand{\ind}{{\rm Ind}}

\title{Mirror symmetry on levels of non-abelian Landau--Ginzburg orbifolds}
\author{Wolfgang Ebeling and Sabir M.~Gusein-Zade
}
\date{}

\begin{document}
\selectlanguage{english}

\maketitle

\begin{abstract}
We consider the Berglund--H\"ubsch--Henningson--Takahashi duality of Landau--Ginzburg orbifolds with a symmetry group generated by some diagonal symmetries and some permutations of variables. We study the orbifold Euler characteristics, the orbifold monodromy zeta functions 
and the orbifold E-functions of such dual pairs. We conjecture that we get a mirror symmetry between these invariants even on each level, where we call level the conjugacy class of a permutation. We support this conjecture by giving partial results for each of these invariants.
\end{abstract}

\setcounter{section}{-1}
\section{Introduction}\label{sec:intro}
Mirror symmetry is the observation originating from physics that there are pairs of objects (originally Calabi--Yau manifolds) with symmetries of certain invariants (originally Hodge numbers).
A well-known instance of mirror symmetry is the Berglund--H\"ubsch--Henningson duality of Landau--Ginzburg orbifolds with abelian groups
of diagonal symmetries. A Landau--Ginzburg orbifold is a pair
$(f,G)$ 
consisting of a quasihomogeneous polynomial $f$ of a special type (a so-called invertible polynomial) and a finite subgroup 
$G$
of the full group of symmetries of $f$ \cite{IV}. In the case that 
$G$
consists only of diagonal symmetries, P.~Berglund, T.~H\"ubsch, and M.~Henningson constructed a dual pair 
$(\widetilde{f},\widetilde{G})$ 
\cite{BH1, BH2}. In a series of papers \cite{EG-MMJ, EG-PEMS, EGT} the authors and A.~Takahashi showed that there are symmetries between certain invariants of such dual pairs such as the orbifold Euler characteristic of the Milnor fibre, the orbifold zeta function 
of the monodromy, and the orbifold E-function.

The object of this paper is to study these symmetries for non-abelian symmetry groups
$G$.
As the group $G$, we consider the semi-direct product of a finite group $G$ of diagonal symmetries and a subgroup $S$ of the group of permutations of the variables.
Takahashi suggested a method to associate a dual pair in this situation. This is called the BHHT-dual pair.  We started the study in \cite{EG-IMRN} where we found that we obtain symmetries only under a special condition on the permutation group called parity condition (PC). One can conjecture that there are natural generalizations of the results for the abelian case to this non-abelian case. Here we give some evidence for the conjecture that the above invariants are not only symmetric for BHHT-dual pairs with PC, but also their distributions among the levels of $S$, namely the conjugacy classes of the elements of $S$.

The plan of the paper is the following. In 
Section~\ref{sec:orb_notions}, we introduce the orbifold invariants which we want to consider, namely, the orbifold Euler characteristic, the orbifold zeta function 
and the orbifold E-function and we discuss relations 
between these invariants.
Section~\ref{sec:duality} is devoted to results 
about these invariants for BHH-dual pairs and the conjectural generalizations to BHHT-dual pairs. In 
Section~\ref{sec:among_levels}, we formulate our main conjectures about the distribution of the invariants among levels. In the remaining sections we indicate some results which support these conjectures. In 
Section~\ref{sec:orb_Euler}, we show that the orbifold Euler characteristics of BHHT-dual pairs coincide up to sign on any level for which the centralizer is cyclic and satisfies PC. In 
Section~\ref{sec:orb_zeta}, we show that, depending on the dimension, the orbifold zeta functions 
of BHHT-dual pairs with PC either coincide or are inverse to each other on the level of an element which generates its centralizer. For the level 1, we show that the mirror map yields a symmetry of  the orbifold E-functions of BHHT-dual pairs under certain conditions, see 
Section~\ref{sec:orb_E_function}. Finally, in 
Section~\ref{sec:example}, we compute the orbifold E-function for a rather simple example. Our conjecture is true for this example, but it turns out to be difficult (if possible) to define an analogue of the mirror map for other levels than 1.

\section{Orbifold Euler characteristic, zeta function,  
and E-function}\label{sec:orb_notions}
Let $X$ be a topological space (good enough, e.g., homeomorphic to a locally closed union of cells in a finite CW-complex)
and let $G$ be a finite group acting on $X$ (also with some goodness properties, e.g., with the action of elements of $G$ being cellular maps). Then the {\em orbifold Euler characteristic} of the pair $(X,G)$ is defined by
$$
\chi^{\orb}(X,G)=\frac{1}{\vert G\vert}
\sum_{{(g_1,g_2)\in G^2:}\atop{g_1g_2=g_2g_1}}
\chi(X^{\langle g_1, g_2\rangle})\,,
$$
where, for a subgroup $H\subset G$, $X^H$ is the fixed point set $\{x\in X : hx=x \mbox{ for all } h \in H\}$ of the
subgroup $H$, $\langle g_1,g_2, \dots\rangle$ is the subgroup
of $G$ generated by the elements $g_1,g_2, \ldots$ (see, e.g., \cite{HH}, \cite{AS}). By $\chi(\cdot)$ we always mean the additive Euler characteristic defined as the alternating sum of dimensions of the cohomology groups $H^q_c(X;\RR)$ with compact support. 
(For complex quasiprojective varieties it coincides with the ``usual'' Euler characteristic: the alternating sum of dimensions of the cohomology groups $H^q(X;\RR)$.)
One has an equivalent equation
for the orbifold Euler characteristic:
$$
\chi^{\orb}(X,G)=
\sum_{[g]\in {\Conj}G}
\chi(X^{\langle g\rangle}/C_G(g))\,,
$$
where ${\Conj}G$ is the set of conjugacy classes of
the elements of $G$, $g$ is a representative of the conjugacy class $[g]$, $C_G(g)$ is the centralizer of $g$ (see, e.g., \cite{HH}).
The {\em reduced orbifold Euler characteristic}
is defined by
$$
\overline{\chi}^{\orb}(X,G)=\chi^{\orb}(X,G)-
\chi^{\orb}({\rm pt},G)\,,
$$
where ${\rm pt}$ is the one-point set with the only
(trivial) action of the group $G$.

For a topological space $X$ and a proper self-map $h:X\to X$,
the zeta function of $h$ is the rational function in $t$
defined by
$$
\zeta_h(t)=\prod_{k\ge 0}
\left(\det({\rm id}-th^*_{\vert H^k_c(X;\RR)})\right)^{(-1)^k}\,,
$$
where $h^*$ is the action of $h$ on the cohomology groups $H^k_c(X;\RR)$ with compact support.

Let $X$ be a complex $n$-dimensional $G$-manifold with a proper $G$-equivariant (i.e., commuting with all the elements of $G$) self-map $h:X\to X$. For an element $g\in G$ and a fixed point
$x\in X^{\langle g\rangle}$ the {\em age} (or the
{\em shift number}) of the element $g$ at $x$ is
${\age_x}(g)=\sum_{i=1}^n \theta_i$, where $\theta_i\in[0,1)$ and
${\mathbf e}[\theta_i]:=\exp{2\pi i \theta_i}$, $i=1, \ldots , n$, are the eigenvalues
of the action of $g$ on the tangent space $T_xX$. The
age of an element is constant on a connected component
of the fixed point set $X^{\langle g\rangle}$. Let
$X^{\langle g\rangle}_{\alpha}$ be the set of points of $X^{\langle g\rangle}$ with the age of $g$ equal to $\alpha$. Then the {\em orbifold zeta function} of
$(h,X,G)$ is defined by
$$
\zeta^{\orb}_{h,X,G}(t)=
\prod_{[g]\in {\Conj}G} \prod_{\alpha \in \QQ}
\zeta_{h_{\vert X^{\langle g\rangle}_{\alpha}/C_G(g)}}
({\mathbf e}[-\alpha]\cdot t)\,.
$$

Let $f(x_1, \ldots, x_n)$ be a quasihomogeneous polynomial:
\begin{equation}\label{eqn:quasihomogeneous}
f(\lambda^{w_1}x_1, \ldots, \lambda^{w_n}x_n)=
\lambda^d f(x_1, \ldots, x_n)
\end{equation}
with the weights $w_1$, \dots, $w_n$ and the quasidegree $d$ being natural numbers.
Let $f$ be endowed with a finite group 
$G\subset GL(n,\CC)$ of linear automorphisms  preserving $f$, let $X$ be the Milnor fiber $V_f=\{f=1\}$ of $f$ and let $h:V_f\to V_f$ be the monodromy transformation
$$
h(x_1, \ldots, x_n)=
\left({\mathbf e}\left[\frac{w_1}{d}\right]x_1, \ldots,
{\mathbf e}\left[\frac{w_n}{d}\right]x_n
\right)\,.
$$
(In the context of Landau--Ginzburg models the monodromy transformation $h$ is often called {\em ``the exponential grading operator''.})
In this case the age of
an element $g\in G$ does not depend on the point
$x\in X^{\langle g\rangle}$ and one has the age function
${\age}:G\to \QQ_{\ge 0}$, where
${\age}(g)=\sum_{i=1}^n\theta_i$, $0\le \theta_i<1$,
with ${\mathbf e}[\theta_i]$ being the eigenvalues of $g$ on $\CC^n$. 
(This coincides with the age of the element $g$ considered as a map $V_f\to V_f$.)
In this case, the definition of the zeta function $\zeta^{\orb}_{f,G}(t)=\zeta^{\orb}_{h,V_f,G}(t)$ of $(f,G)$ was given in~\cite{EG-PEMS}. The {\em reduced orbifold zeta function} of the monodromy
is defined by
$$
\overline{\zeta}^{\orb}_{f,G}(t)=
\zeta^{\orb}_{f,G}(t)\left/ \prod_{[g]\in {\rm Cong\,}G}(1-{\mathbf e}[-{\age}(g)]t)   \right.{.}
$$
One can see that the (reduced) orbifold Euler characteristic of $(V_f,G)$ is the degree of the (reduced) orbifold zeta function 
as a rational function in $t$ (that is the degree of the numerator minus the degree of the denominator).

One can say that the orbifold zeta function of the monodromy
is the zeta function of the action of the monodromy transformation on the Fan--Jarvis--Ruan quantum cohomology group associated to $(f,G)$ (\cite{FJR}). A generating polynomial of the bigrading on this space is the orbifold E-function defined for an invertible polynomial with an abelian group of (diagonal) symmetries in~\cite{EGT}. Here we give a general definition for an arbitrary quasihomogeneous polynomial $f$ and for an arbitrary $G$ as it was made in~\cite{ET} for abelian subgroups $G \subset {\rm SL}(n,\CC)$ acting diagonally on the coordinates.

Let $f(x_1, \ldots, x_n)$ be a quasihomogeneous polynomial with an isolated critical point at the origin and let $G\subset GL(n,\CC)$ be a finite group
of transformations preserving $f$.
The cohomology group $H^{n-1}(V_f;\CC)$ is endowed with the mixed Hodge structure introduced by J.~H.~M. Steenbrink
in~\cite{Steenbrink}. Let us consider the $\QQ\times\QQ$-graded vector space
$\H_f=\displaystyle\bigoplus_{p,q\in \QQ}\H_f^{p,q}$ defined by
\begin{enumerate}
 \item[1)] if $p+q\ne n$, then $\H_f^{p,q}:=0$;
 \item[2)] if $p+q=n$, then $\H_f^{p,q}:=
 {\rm Gr}_{F^{\bullet}}^{[p]}H^{n-1}(V_f;\CC)_{{\mathbf e}[p]}$,
 where $H^{n-1}(V_f;\CC)_{{\mathbf e}[p]}$ is the eigenspace
 of the monodromy transformation $h$ corresponding to the eigenvalue ${\mathbf e}[p]$, $[p]$ is the integer part of $p$.
\end{enumerate}
If $n=1$, one has to consider the reduced cohomology group $\overline{H}^{n-1}(V_f;\CC)$. If $n=0$ (this case is important for the definition below), we assume
$H^{-1}(V_f;\CC)$ to be one dimensional with the trivial
action of $G$. This means that one considers
the ``critical point'' of the function of zero variables to be non-degenerate and thus to have the dimension of the cohomology group in dimension $(-1)$ to be equal to 1.

For any $g\in G$, the restriction $f^g$ of the polynomial $f$ to the fixed point set
$(\CC^n)^{\langle g\rangle}$ has an isolated critical point as well: see, e.g., 
\cite{Wall}. 
Let $n_g$ be the dimension of the fixed locus $(\CC^n)^{\langle g\rangle}$.
The centralizer $C_G(g)$ of the element $g$ acts on $(\CC^n)^{\langle g\rangle}$ and on the Milnor fiber of $f^g$. For $i=0,1$, let
$$
\H_{f,G,\overline{i}}:=
\bigoplus_{{[g]\in {\rm Conj\,}G}\atop{n_g\equiv i \bmod 2}}(\H_{f^g})^{C_G(g)}(-\age(g),-\age(g))\,,
$$
where $(\H_{f^g})^{C_G(g)}$ is the $C_G(g)$-invariant
part of $\H_{f^g}$, $(-{\age}(g),-{\age}(g))$
means the shift of the bigrading ($\overline{0}$ means {\em even}, $\overline{1}$ means {\em odd}).
Let
$$
E_{\overline{i}}(f,G)(t,\overline{t}):=
\sum_{p,q\in \QQ}
\dim_{\CC}(\H_{f,G,\overline{i}})^{p,q}
 t^{p-\frac{n}{2}}\overline{t}^{q-\frac{n}{2}}\,,
$$
where $t$ and $\overline{t}$ are two (independent) variables.

\begin{definition}
 The {\em (orbifold) {\rm E}-function} of $(f,G)$ is
 $$
 E(f,G)(t,\overline{t})=
 E_{\overline{0}}(f,G)(t,\overline{t})-
 E_{\overline{1}}(f,G)(t,\overline{t}).
 $$
\end{definition}

For an exponent $s:=q-\frac{n}{2}$ of the variable $\overline{t}$ in $E(f,G)(1,\overline{t})$ with non-zero coefficient,
${\mathbf e}[q]=(-1)^n{\mathbf e}[s]$ is an eigenvalue of the transformation induced by the monodromy transformation $h$ on 
$(H^{n_g-1}(V_{f^g};\CC))^{C_G(g)}$ multiplied by ${\mathbf e}[-\age(g)]$.
It gives the summand $(-1)^{n_g}\overline{t}^s$
in $E(f,G)(1,\overline{t})$ and the factor
$(1-(-1)^n{\mathbf e}[s]t)^{(-1)^{n_g-1}}$ in 
$\overline{\zeta}_{f,G}^{\orb}(t)$.
This implies that, if
$$
E(f,G)(1,\overline{t})=\sum_{s\in \QQ}a_s\overline{t}^s \quad (a_s \in \ZZ),
$$
then
$$
\overline{\zeta}_{f,G}^{\orb}(t)=
\prod_{s\in \QQ}(1-(-1)^n{\mathbf e}[s]t)^{-a_s}.
$$

In particular, the orbifold $E$-function determines
the orbifold monodromy zeta-function.

\begin{remark} As it was mentioned above, the orbifold $E$-function is a generating polynomial of the bigrading on the quantum
cohomology group of Fan--Jarvis--Ruan. One can consider other polynomials, e.g., 
$$E'(f,G)(t,\overline{t})=E_{\overline{0}}(f,G)(t,\overline{t})+
E_{\overline{1}}(f,G)(t,\overline{t})$$
defined in \cite{EG-MM} or the Poincar\'e polynomial defined in~\cite{Mukai2}. 
\end{remark} 

\section{Duality}\label{sec:duality}
An attempt to construct mirror symmetric pairs led
Berglund, H\"ubsch, and Henningson to the following concept. A quasihomogeneous polynomial
$f(x_1,\ldots,x_n)$ 
is called {\em invertible} if the number of monomials in it (with non-zero coefficients) is equal to the number $n$ of variables, i.e.,
\begin{equation}\label{eqn:invertible}
f(x_1,\ldots,x_n)=
\sum_{i=1}^n c_i\prod_{j=1}^n x_j^{E_{ij}}
\end{equation}
with $c_i\ne 0$, $E_{ij}\in\ZZ_{\ge 0}$, and the matrix
$E=(E_{ij})$ is invertible (over $\QQ$, i.e.,
$\det E\ne 0$). Without loss of generality one may assume that $c_i=1$ for all $i$ (this can be achieved by rescaling the variables).

\begin{remark}
 Usually, in the definition of an invertible polynomial, it is assumed in advance that
 an invertible polynomial
 is non-degenerate, i.e., has an isolated critical point at the origin.
 We prefer not to include this assumption into the initial
 definition since some statements about invertible
 polynomials and symmetries of their invariants
 hold without it
 (e.g., Equations~(\ref{eqn:MMJ}) and~(\ref{eqn:PEMS}) below). However, in this paper we shall only consider non-degenerate invertible polynomials dropping the adjective non-degenerate.
\end{remark}

One can show that a polynomial is invertible if and only if it is the Sebastiani--Thom sum (i.e., sum of polynomials in non-intersecting sets of variables)
of polynomials of the form:
\begin{enumerate}
 \item[1)] Fermat type: $x_1^{a_1}$, $a_1\ge 2$;
 \item[2)] chain type:
 $x_1^{a_1}x_2+x_2^{a_2}x_3+\cdots+x_m^{a_m}$, $a_m\ge 2$;
 \item[3)] loop type:
 $x_1^{a_1}x_2+x_2^{a_2}x_3+\cdots+x_m^{a_m}x_1$
(if $m$ is even, then neither $a_1=a_3=\ldots=a_{m-1}=1$ nor
 $a_2=a_4=\ldots=a_{m}=1$).
\end{enumerate}

\begin{remark}
a) Depending on the desire, the Fermat type polynomial can be considered as a particular case either of chain type or of loop type.\\
b) If $a_1=1$ in 1) or $a_m=1$ in 2), the corresponding polynomial
 is regular (does not have a critical point) at the origin.\\
c) If, in 3), the restriction in parentheses does not hold,
 the polynomial either is not quasihomogeneous (if there exists $a_i$ different from $1$, in~(\ref{eqn:quasihomogeneous}) some weights $w_i$ are forced to be equal to zero) or is not invertible
 (if all $a_i$ are equal to $1$, $\det E=0$).
\end{remark}

The group of diagonal symmetries of an invertible polynomial $f(x_1,\ldots,x_n)$ is
$$
G_f=\{(\lambda_1, \ldots, \lambda_n)\in(\CC^*)^n:
f(\lambda_1 x_1,\ldots,\lambda_n x_n)=f(x_1,\ldots,x_n)\}\,.
$$
One can show that $\vert G_f\vert=\vert\det E\vert$.

The {\em dual} (or the {\em transpose}) of the invertible polynomial~(\ref{eqn:invertible}) is
$$
\widetilde{f}(x_1,\ldots,x_n)=
\sum_{i=1}^n \prod_{j=1}^n x_j^{E_{ji}}.
$$
One can show that there is a canonical isomorphism between the group $G_{\widetilde{f}}$ of diagonal symmetries of $\widetilde{f}$ and the group
${\rm Hom}(G_f,\CC^*)$ of characters of $G_f$ (see details, e.g., in~\cite{EG-BLMS}). For a subgroup $G\subset G_f$, its
dual $\widetilde{G}$ is the subgroup of $G_{\widetilde{f}}$ consisting of the characters of $G_f$ vanishing (i.e., having the value $1$) on $G$.
The pair $(\widetilde{f}, \widetilde{G})$ is called
{\em Berglund--H\"ubsch--Henningson}- ({\em BHH}- for short) {\em dual} to the pair $(f,G)$.

In \cite{EG-MMJ, EG-PEMS, EGT} it was shown that, for BHH-dual pairs $(f,G)$ and $(\widetilde{f}, \widetilde{G})$, one has
\begin{eqnarray}
 \overline{\chi}^{\orb}(V_f,G) & = & (-1)^n\,
 \overline{\chi}^{\orb}(V_{\widetilde{f}},\widetilde{G})\,; \label{eqn:MMJ}\\
 \overline{\zeta}^{\orb}_{f,G}(t) & = &
 \left(\overline{\zeta}^{\orb}_{\widetilde{f},\widetilde{G}}(t)\right)^{(-1)^n}\, ; \label{eqn:PEMS} \\
 E(f,G)(t,\overline{t}) & = & (-1)^n
 E(\widetilde{f},\widetilde{G})(t^{-1},\overline{t})\,. \label{eqn:EGT}
\end{eqnarray}
Usually similar invariants of pairs are considered under
a (physically-motivated) assumption on  the group $G$ to be ``admissible''. Note that the
relations~(\ref{eqn:MMJ})~--(\ref{eqn:EGT}) hold for an arbitrary subgroup $G\subset G_f$ (without the assumption to be admissible).

The group $S_n$ of permutations on $n$ elements acts on the space $\CC^n$ by permuting the variables. Let $f$ be an invertible polynomial in $n$ variables and $G$ a group of diagonal
symmetries of $f$ (i.e., $G\subset G_f$) and let $S$ be
a subgroup of $S_n$ preserving $f$ and $G$
($\sigma^{-1}G\sigma=G$ for 
$\sigma\in S$). The action of $S$ on $G$ (by conjugation)
defines the semidirect
product $G\rtimes S$ with its action on $\CC^n$.
(The precise equations for the group structure on 
$G\rtimes S$ and for its action on $\CC^n$ can be found, e.g., in \cite[Remark 2.1]{EG-SIGMA}.) In this case the dual polynomial $\widetilde{f}$ and the dual subgroup
$\widetilde{G}\subset G_{\widetilde{f}}$ are also preserved by $S$ and, for the pair $(f,G\rtimes S)$,
one gets the so-called {\em Berglund--H\"ubsch--Henningson--Takahashi-} ({\em BHHT}- for short) {\em dual} pair
$(\widetilde{f},\widetilde{G}\rtimes S)$: see~\cite{EG-IMRN}. In~\cite{EG-IMRN}, it was shown that, for BHHT-dual pairs, the obvious analogues of 
the symmetries (\ref{eqn:MMJ})~-- (\ref{eqn:EGT}) do not hold in general, but may hold only under a special condition on the subgroup $S\subset S_n$ (called ``PC'': parity condition),
namely, for any subgroup
$T\subset S$, $\dim(\CC^n)^T\equiv n \bmod 2$.

One can conjecture that, under PC,
one has 
symmetries like (\ref{eqn:MMJ}), (\ref{eqn:PEMS}), and (\ref{eqn:EGT}), i.e.,
\begin{eqnarray}
 \overline{\chi}^{\orb}(V_f,G\rtimes S) & = & (-1)^n\,
 \overline{\chi}^{\orb}(V_{\widetilde{f}},\widetilde{G}\rtimes S)\,; \label{eqn:NC-MMJ} \\
 \overline{\zeta}^{\orb}_{f,G\rtimes S}(t) & = &
 \left(\overline{\zeta}^{\orb}_{\widetilde{f},\widetilde{G}\rtimes S}(t)\right)^{(-1)^n}\,; \label{eqn:NC-PEMS} \\
 E(f,G\rtimes S)(t,\overline{t}) & = & (-1)^n
 E(\widetilde{f},\widetilde{G}\rtimes S)(t^{-1},\overline{t})\,. \label{eqn:NC-EGT}
\end{eqnarray}

Equation~(\ref{eqn:NC-MMJ}) was proved in \cite{EG-PAMQ} for polynomials of loop type with $S$ satisfying PC and in \cite{EG-SIGMA} for a cyclic subgroup $S$ (also satisfying PC).

From the computations in~\cite{Priddis} one can derive the symmetries (\ref{eqn:NC-MMJ}) and (\ref{eqn:NC-PEMS}) for the following pairs $(f,G\rtimes S)$ (with Fermat type polynomials $f$):
\begin{enumerate}
 \item[1)] $f=x_1^4+\cdots +x_4^4$, $G=\langle J\rangle$, $S=\langle(123)\rangle$, where
 $J=({\mathbf e}[1/4], \ldots, {\mathbf e}[1/4])$ is the exponential grading operator (the monodromy transformation);
 \item[2)] $f=x_1^5+\cdots +x_5^5$, $G=\langle J\rangle$, $S=\langle(12)(34)\rangle$, where
 $J=({\mathbf e}[1/5], \ldots, {\mathbf e}[1/5])$.
\end{enumerate}

A symmetry similar to~(\ref{eqn:NC-EGT}) was proved
in~\cite{Mukai2} for the pairs $(f,G\rtimes S)$ with 
$f=x_1^{p-1}x_2+\cdots+x_p^{p-1}x_{1}$ (a loop), an arbitrary subgroup $G$ of $G_f\cong \ZZ_{(p-1)^p+1}$ 
and $S=\langle(1 2 \ldots p)\rangle$, where $p$ is an odd prime number. 
Moreover, one can derive from \cite{BI} a symmetry similar to the symmetry of \cite{Mukai2} for pairs $(f,G\rtimes S)$ with $f=x_1^p+ \cdots +x_p^p$ (a Fermat polynomial), $G$ is either $G_f$ or $G_f \cap {\rm SL}(p,\CC)$, $S$ satisfies PC, and $p$ is a prime number.

\section{Distribution of invariants among levels}\label{sec:among_levels}
Let $f$ be an invertible polynomial in $n$ variables and 
let $G\rtimes S$ be a group of its symmetries ($G\subset G_f$, $S\subset S_n$). The conjugacy class of an element
$(\underline{\lambda}, \sigma)\in G\rtimes S$ is
contained in $G\times [\sigma]$, where $[\sigma]\subset S$ is the conjugacy class of $\sigma$.
Therefore each orbifold invariant considered in Section~\ref{sec:orb_notions} is the sum or the product of
the parts corresponding to different conjugacy classes
of the elements of $S$. Namely,
$$
\chi^{\orb}(V_f, G\rtimes S)=\sum_{[\sigma]\in {\Conj S}}\chi^{\orb,[\sigma]}(V_f, G\rtimes S)\,.
$$
where
$$
\chi^{\orb,[\sigma]}(V_f, G\rtimes S)=
\sum_{{[g]\in {\Conj}G \rtimes S:}\atop{[g]\subset G\times [\sigma]}}
\chi(V_f^{\langle g\rangle}/C_G(g))\,;
$$
and
$$
\overline{\zeta}^{\orb}_{f, G\rtimes S}(t)=
\prod_{[\sigma]\in {\Conj S}}\overline{\zeta}^{\orb,[\sigma]}_{f, G\rtimes S}(t)\,.
$$
where
$$
\overline{\zeta}^{\orb,[\sigma]}_{f, G\rtimes S}(t)=
\prod_{[g]\in {\Conj}G \rtimes S: \atop [g]\subset G\times [\sigma]}\overline{\zeta}_{ h_{\vert V_{f^g}/C_G(g)}}
({\mathbf e}[-\age(g)]t) \,.
$$
Finally 
$$
E(f,G\rtimes S)(t,\overline{t})=\sum_{[\sigma]\in\Conj S}
E^{[\sigma]}(f,G\rtimes S)(t,\overline{t})\,,
$$
where
$$
E^{[\sigma]}(f,G\rtimes S)(t,\overline{t})=
 \sum_{p,q\in \QQ}
 \left(\dim_{\CC}(\H^{[\sigma]}_{f,G\rtimes S,\overline{0}})^{p,q}-
 \dim_{\CC}(\H^{[\sigma]}_{f,G\rtimes S,\overline{1}})^{p,q}\right)
 t^{p-\frac{n}{2}}\overline{t}^{q-\frac{n}{2}}\,,
$$
$$
\H^{[\sigma]}_{f,G\rtimes S,\overline{i}}:=
\bigoplus_{{[g]\in {\rm Conj\,}G\rtimes S, [g]\subset G\times[\sigma]}\atop{n_g\equiv i \bmod 2}}(\H_{f^g})^{C_{G \rtimes S}(g)}(-\age(g),-\age(g))\,.
$$

There are some indications for the following conjectures ($S$ satisfying PC). 

\begin{conjecture}\label{conj:chi}
 \begin{equation}\label{eqn:conj_chi}
 \chi^{\orb,[\sigma]}(V_f, G\rtimes S)=(-1)^n
 \chi^{\orb,[\sigma]}(V_{\widetilde{f}}, \widetilde{G}\rtimes S)\,.
 \end{equation}
\end{conjecture}

\begin{conjecture}\label{conj:zeta}
 \begin{equation}\label{eqn:conj_zeta}
 \overline{\zeta}^{\orb,[\sigma]}_{f, G\rtimes S}(t)=
 \left(\overline{\zeta}^{\orb,[\sigma]}_{\widetilde{f}, \widetilde{G}\rtimes S}(t)\right)^{(-1)^n}.
 \end{equation}
\end{conjecture}

\begin{conjecture}\label{conj:E}
 \begin{equation}\label{eqn:conj_E}
 E^{[\sigma]}(f,G\rtimes S)(t,\overline{t})=(-1)^n
 E^{[\sigma]}(\widetilde{f},\widetilde{G}\rtimes S)(t^{-1},\overline{t})\,.
 \end{equation}
\end{conjecture}

We shall call {\em levels} the conjugacy classes of the elements of $S$. The conjectures state that not only the discussed invariants are symmetric for BHHT-dual pairs, but
also their distributions among the levels.
In what follows, we shall show some examples supporting Conjectures~\ref{conj:chi}--\ref{conj:E}.

\section{Orbifold Euler characteristic}\label{sec:orb_Euler}
The proof of Theorem~4.1 in \cite{EG-SIGMA} gives equation
(\ref{eqn:conj_chi}) for a cyclic group $S$ with PC. Note that in this case the conjugacy class $[\sigma]$ in $S$ consists only of $\sigma$ itself.
Here we prove a somewhat stronger statement. Let $f$, $G$, and $S$ be as above.

Suppose that $K\subset H$ are finite groups.
One has an operation $\ind_K^H$ which converts $K$-spaces to
$H$-spaces. Namely, for a $K$-space $X$, the $H$-space 
$\ind_K^H X$ is $(H\times X)/\sim$, where the equivalence relation $\sim$ is defined by $(h_1,x_1)\sim(h_2,x_2)$ if there exists
$g\in K$ such that $h_1=h_2g^{-1}$, $x_1=gx_2$; the action of
$H$ on $\ind_K^H X$ is defined in a natural way: $h'(h,x)=(h'h,x)$. (A useful example:
for $G\subset K\subset H$, one has $\ind_K^H K/G=H/G$.
In particular, $\ind_K^H K/K=H/K$.)

\begin{theorem}\label{theo:chi}
Assume that $C_S(\sigma)$ is cyclic and satisfies PC (that is, it is contained in $A_n$). Then Equation~(\ref{eqn:conj_chi}) holds.
\end{theorem}

\begin{proof}
 One has the decompositions
 $$
 \CC^n\setminus\{0\}
 =\coprod_{I\subset I_0,\, I\ne\emptyset}
 (\CC^*)^I,\quad V_f=\coprod_{I\subset I_0,\, I\ne\emptyset}
 V_f^I\,,
 $$
 where $I_0=\{1,\ldots, n\}$, $V_f^I=V_f\cap (\CC^*)^I$.
 Using the action of $S$ on $2^{I_0}$, one can write
 $$
 V_f=\coprod_{\calJ\in (2^{I_0}\setminus{\{\emptyset\}})/S}   {\ }\coprod_{J\in\calJ}V_f^J\,.
 $$
 Let $S^I\subset S$ be the isotropy subgroup of $I$ for the $S$-action on $2^{I_0}$. For any $I$ from $\calJ\in (2^{I_0}\setminus{\{\emptyset\}})/S$, one has
 $$
 \left(\coprod_{J\in\calJ}V_f^J, G\rtimes S\right)=
 {\rm Ind}_{G\rtimes S^I}^{G\rtimes S}(V_f^I,G\rtimes S^I)\,.
 $$
 
 Elements of the form $(\underline{\lambda}, \sigma)$
 may have fixed points in
 $\coprod\limits_{J\in\calJ}(\CC^*)^J$ only if
 $\sigma\in S^I$ for some $I\in\calJ$ and, in this case,
 $$
 \chi^{\orb, [\sigma]}({\rm Ind}_{G\rtimes S^I}^{G\rtimes S}(V_f^I,G\rtimes S^I)=
 \chi^{\orb, [\sigma]}(V_f^I, G\rtimes S^I)\,.
 $$
 (This is a simple version of Theorem~1 from \cite{GLM}; cf. Theorem~\ref{theo:induction} below.)
 
 For $I\subset I_0$, let $f^I$ be the restriction of $f$ to $\CC^I$. If $f^I$ has less than $\vert I\vert$ monomials, then $\chi^{\orb, [\sigma]}(V_f^I, G\rtimes S^I)=0$. This follows from the fact that, in this case,
 there exists a free $\CC^*$-action on
$(\CC^*)^I$
 preserving $f^I$ and commuting with the action of
 $G_{f^I}\rtimes S^I$ (cf.~\cite[Lemma~1]{EG-IMRN}).
 If $f^I$ has less than $\vert I\vert$ monomials, then
 $\widetilde{f}^{\overline{I}}$ has less than
 $\vert\overline{I}\vert$ monomials
($\overline{I}:=I_0\setminus I$). Therefore in this case
 $$
 \chi^{\orb, [\sigma]}(V_f^I, G\rtimes S^I)=
 \chi^{\orb, [\sigma]}(V_{\widetilde{f}}^{\overline{I}}, \widetilde{G}\rtimes S^{\overline{I}})=0\,.
 $$
 (Pay attention that $S^{\overline{I}}=S^I$.)
 
 Now let $f^I$ have $\vert I\vert$ monomials (and thus be invertible).
 The arguments in the proof of \cite[Theorem~4.1]{EG-SIGMA}
 give
 $$
 \chi^{\orb, [\sigma]}(V_f^I, G\rtimes S^I)=
 \frac{1}{\vert S^I\vert}
 \sum_{{\sigma'\in S^I}\atop{\sigma\sigma'=\sigma'\sigma}}
 \chi^I_{f,G}(\sigma,\sigma')\,,
 $$
 where
 $$
  \chi_{f,G}^I(\sigma,\sigma')=\frac{1}{\vert G\vert}
 \sum_{{(\underline{\lambda},\underline{\lambda}')\in G^2:}\atop
 {(\underline{\lambda},\sigma)(\underline{\lambda}',\sigma')=
 (\underline{\lambda}',\sigma')(\underline{\lambda},\sigma)}}
 \chi\left((V_f^I)^{\langle(\underline{\lambda},\sigma),(\underline{\lambda}',\sigma')\rangle}\right).
 $$ 
 If one defines 
 $$\chi_{f,G}^{\emptyset}(\sigma,\sigma'):=-\frac{1}{\vert G\vert}\vert\{(\underline{\lambda},\underline{\lambda}')
 \in G^2:
(\underline{\lambda},\sigma)(\underline{\lambda}',\sigma')=(\underline{\lambda}',\sigma')(\underline{\lambda},\sigma)
\}\vert,
$$
one has
$$
\overline{\chi}^{\orb, [\sigma]}(V_f, G\rtimes S)=
\sum_{\calJ=[I]\in 2^{I_0}/S: \atop \sigma \in S^I} \frac{1}{|S^I|} \sum_{\sigma' \in S^I: \atop \sigma \sigma' = \sigma' \sigma}   \chi_{f,G}^I(\sigma,\sigma') \, .
$$

Let $s$ be the generator of the centralizer $C_{S^I}(\sigma) \subset C_S(\sigma)$ and let $\sigma=s^m$, $\sigma'=s^{m'}$. Then \cite[Proposition~4.2]{EG-SIGMA} gives
\[
\chi_{f,G}^I(\sigma,\sigma') =   \chi_{f,G}^I(s^{m^*},1)
\]
with $m^*= \gcd(m,m', |C_{S^I}(\sigma)|)$.
 
The fact that, for all $I$ (including $I=\emptyset$),
\[
\chi_{f,G}^I(\sigma,1)= (-1)^n
 \chi_{\widetilde{f},\widetilde{G}}^{\overline{I}}(\sigma,1)
\]
is shown in the proof of \cite[Theorem~4.1]{EG-SIGMA} since \cite[Proposition~4.3]{EG-SIGMA} and Equation \cite[(4.5)]{EG-SIGMA} do not demand $S$ to be cyclic. The fact that $S$ is cyclic was used in \cite{EG-SIGMA} only for the reduction of $\chi^I_{f,G}(\sigma, \sigma')$ to $\chi_{f,G}^I(s^{m^*},1)$ (\cite[Proposition~4.2]{EG-SIGMA}). 
\end{proof}

\begin{corollary} If $S \subset S_n$ (satisfying PC) is such that the centralizer of each non-unit element is cyclic, then
\begin{equation}
 \overline{\chi}^{\orb}(V_f,G\rtimes S)  =  (-1)^n\cdot
 \overline{\chi}^{\orb}(V_{\widetilde{f}},\widetilde{G}\rtimes S)\,. \label{eqn:cor}
\end{equation}
\end{corollary}

\begin{proof}
Equation (\ref{eqn:cor}) follows from its analogue for all the levels $[\sigma] \in {\rm Conj}\, S$. For $\sigma \neq 1$, this is just Theorem~\ref{theo:chi}. For $\sigma=1$ this is a direct consequence of Equation \cite[(4.5)]{EG-SIGMA} (since $\chi^I_{f,G}(1,\sigma')=\chi^I_{f,G}(\sigma',1)$).
\end{proof}

\begin{remark} \label{rem:cent_cyclic}
One can indicate the following examples of groups $S$ possessing the mentioned property, i.e., the centralizer of each non-unit element is cyclic:
\begin{enumerate}
\item[1)] $S$ is cyclic.
\item[2)] $S=D_k$ (the dihedral group) with $k$ odd. (Pay attention that $D_3=S_3$, but in order to satisfy PC, the number of variables $n$ should not be 3.)
\item[3)] Semidirect products $\ZZ_p \rtimes \ZZ_q$ with prime $p$ and $q$ such that $p \equiv 1 \, {\rm mod}\, q$. (There is only one non-abelian group of order $pq$ in this case.)
\end{enumerate}
\end{remark}

\section{Orbifold zeta function} \label{sec:orb_zeta} 
Let $f$, $G$, and $S$ be as above, $S$ satisfies PC.

\begin{theorem}\label{theo:orb_zeta}
 If $\sigma\in S$ is such that the centralizer $C_S(\sigma)$ is generated by $\sigma$, then
 \begin{equation}\label{eqn:orb_zeta}
  \overline{\zeta}^{\orb,[\sigma]}_{f,G \rtimes S}(t)=
  \left(\overline{\zeta}^{\orb,[\sigma]}_{\widetilde{f},\widetilde{G} \rtimes S}(t)\right)^{(-1)^n}\,.
 \end{equation}
\end{theorem}

\begin{remark} Among the groups listed in Remark~\ref{rem:cent_cyclic}, those satisfying the condition of Theorem~\ref{theo:orb_zeta} are
\begin{enumerate}
\item[1)] cyclic groups of prime order;
\item[2)] the dihedral group $D_k$ with $k$ prime;
\item[3)] the non-abelian semidirect products $\ZZ_p \rtimes \ZZ_q$.
\end{enumerate}
\end{remark}

We shall use the following generalization
of~\cite[Theorem~1]{GLM}. Let $G$ be a subgroup of a finite group $K$, let $X$ be a complex $G$-manifold with a proper self-homeomorphism $h:X\to X$ of finite order
commuting with actions of the elements of $G$. The action of $h$ on $X$ can in a natural way be extended to
$\widehat{h}:\ind_G^K X\to\ind_G^K X$
($\widehat{h}(k,x)=(k,h(x))$ for $k\in K$, $x\in X$).

\begin{theorem}\label{theo:induction}
 One has
\begin{equation}\label{eqn:induction}
\overline{\zeta}^{\orb}_{h, X ,G}(t)=
\overline{\zeta}^{\orb}_{\widehat{h}, \ind_G^K X ,K}(t) \, .
 \end{equation}
\end{theorem}

\begin{proof}
 To use the arguments similar to those in~\cite[Lemma~1 \& Theorem~1]{GLM}, let us consider triples $(X, \{G_i\}, h)$ consisting of a (topological) space $X$ with actions of different finite groups $G_i$ on different parts and with a self-homeomorphism $h$ of finite order commuting with all the actions. This means that $X$ is a finite disjoint union $X=\coprod_{i=1}^s X_i$ with the finite group $G_i$ acting on $X_i$ and $h:X_i\to X_i$ is such that $g\circ h=h\circ g$ for any $g\in G$ (cf.~\cite[Definition~2]{GLM}). We shall write $(X, \{G_i\}, h)=
 \coprod_{i=1}^s (X_i, G_i, h)$. Two triples
 $(X, \{G_i\}, h)$ and $(Y, \{G_i^{'}\}, h')$ are called equivalent if there exist partitions
 $(X, \{G_i\}, h)=\coprod_{j=1}^N (X_{(j)}, G_{(j)}, h)$
 and $(Y,\{G_i^{'}\}, h')=\coprod_{j=1}^N (Y_{(j)}, G_{(j)}^{'}, h')$ such that there exist 
 homeomorphisms $\psi_j: X_{(j)} \to Y_{(j)}$ and isomorphisms 
$\phi_j: G_{(j)} \to G'_{(j)}$ such that $\psi_j(gx)=\phi_j(g)\psi_j(x)$ and $(h' \circ \psi_j)(x)=(\psi_j \circ h)(x)$ for $x \in X_{(j)}$ and $g \in G_{(j)}$.

Let $X,G,K,h$ be as above. An obvious version of \cite[Lemma~1]{GLM} says that, for $g \in G$, 
\begin{equation}
((\ind_G^K X)^{\langle g \rangle}, C_K(g),h) \mbox{ and } \coprod_{[g'] \in {\rm Conj} \, G: \atop [g']_K = [g]_K} (\ind_{C_G(g')}^{C_K(g')} X^{\langle g' \rangle}, C_K(g'),\widehat{h}) \label{eqn:trip}
\end{equation} 
are equivalent
($[g]_K$ is the conjugacy class of $g$ in $K$).

It is obvious that 
\[
\zeta_{\widehat{h}_{\vert(\ind_G^K X/K)}}(t)=\zeta_{h_{\vert X/G}}(t) \, .
\]
One has 
\[
\zeta^{\orb}_{h,\ind_G^K X,K}(t) = \prod_{\alpha \in \QQ} \prod_{[k] \in {\rm Conj}\, K} \zeta_{h_{\vert (\ind_G^K X)^{\langle k \rangle}_\alpha/C_K(k)}}({\mathbf e}[-\alpha]t) \, .
\]
It is easy to see that, for $k\in K$,
$(\ind_G^K X)^{\langle k \rangle}$ is not empty if (and only if) there exists $g \in G$ such that $[g]_K=[k]_K$ (and $X^{\langle g \rangle}$ is not empty). Equation~(\ref{eqn:trip}) implies that
\begin{eqnarray*}
\zeta^{\orb}_{h,\ind_G^K X,K}(t)  & = &  \prod_{\alpha \in \QQ} \prod_{[g'] \in {\rm Conj}\, G} \zeta_{h_{\vert (\ind^{C_K(g')}_{C_G(g')} X)^{\langle g' \rangle}_\alpha/C_K(g')}}({\mathbf e}[-\alpha]t)\\
& = &  \prod_{\alpha \in \QQ} \prod_{[g'] \in {\rm Conj}\, G} \zeta_{h_{\vert X^{\langle g' \rangle}_\alpha/C_K(g')}}({\mathbf e}[-\alpha]t) \\
& = & \zeta^{\orb}_{h,X,G}(t) \,.
\end{eqnarray*}
\end{proof}

\begin{proof*}{Proof of Theorem~\ref{theo:orb_zeta}}
We shall use the decomposition 
\[ V_f=\coprod_{I\subset I_0,\, I\ne\emptyset} V_f^I
\]
described in the proof of Theorem~\ref{theo:chi}. (The monodromy operator $h$ respects this decomposition.)
Therefore
\[
\zeta^{\orb,[\sigma]}_{h,V_f,G \rtimes S}(t)= \prod_{\calJ\in (2^{I_0}\setminus{\{\emptyset\}})/S} \zeta^{\orb,[\sigma]}_{h,\coprod_{J \in \calJ}V_f^J,G \rtimes S}(t) \, . 
\]
Theorem~\ref{theo:induction} together with the remark that the parts of the zeta functions in Equation~(\ref{eqn:induction}) corresponding to different levels coincide gives
\[
\zeta^{\orb,[\sigma]}_{h,V_f,G \rtimes S}(t)= \prod_{\calJ\in (2^{I_0}\setminus{\{\emptyset\}})/S} \zeta^{\orb,[\sigma]}_{h,V_f^I,G \rtimes S^I}(t) 
\]
for a representative $I$ of $\calJ$.

To prove Equation~(\ref{eqn:orb_zeta}), it is sufficient to show that
\[
\zeta^{\orb,[\sigma]}_{h,V_f^I,G \rtimes S^I}(t) = \left( \zeta^{\orb,[\sigma]}_{\widetilde{h},V_{\widetilde{f}}^{\overline{I}},\widetilde{G} \rtimes S^{\overline{I}}}(t) \right)^{(-1)^n}
\]
for $\sigma \in S^I=S^{\overline{I}}$, $\sigma \neq 1$. We shall show that
\[
\zeta^{\orb,[\sigma]}_{h,V_f^I,G \rtimes S^I}(t) = (1-t^{\ell_I})^{R_I}
\]
for some exponent $R_I$, where the exponent $\ell_I$ is the one given by Equation \cite[(4.6)]{EG-PEMS}, i.e., the exponent in the equation of the orbifold zeta function 
on $V_f^I$ for the group $G$ itself (i.e., for trivial $S^I$).

Let $(\underline{\lambda},\sigma) \in G \rtimes S^I$, $\sigma \neq 1$. First let us show that 
\[ 
\zeta_{h_{\vert (V_f^I)^{\langle (\underline{\lambda}, \sigma)\rangle}/C_{G \rtimes S}(\underline{\lambda},\sigma)}}(t) = (1-t^{m_I})^{R'_I}
\]
for some exponent $R'_I$, where $m_I$ is the same exponent as in Equation \cite[(4.3)]{EG-PEMS}. The centralizer of the element $(\underline{\lambda},\sigma)$ is generated by the elements $(\underline{\nu},1)$ commuting with $(\underline{\lambda}, \sigma)$ and the element $(\underline{\lambda}, \sigma)$ itself. (This follows from the condition that $C_S(\sigma)=\langle \sigma \rangle$.) The element $(\underline{\nu}, 1)$ commutes with $(\underline{\lambda}, \sigma)$ if and only if $\underline{\nu} \in {\rm Ker}\, A_\sigma$, where $A_\sigma(\underline{\mu}):=\underline{\mu} (\sigma(\underline{\mu}))^{-1}$ (see \cite[Definition~3.2]{EG-SIGMA}). Therefore
\[
C_{G \rtimes S}(\underline{\lambda}, \sigma)= \langle {\rm Ker}\, A_\sigma \cap G, (\underline{\lambda}, \sigma) \rangle \, .
\]
The monodromy operator $h$ acts on $(V_f^I)^{\langle (\underline{\lambda}, \sigma)\rangle}$. The exponent in the binomial $(1-t^\bullet)$ in $\zeta_{h_{\vert (V_f^I)^{\langle (\underline{\lambda}, \sigma)\rangle}/C_{G \rtimes S}(\underline{\lambda},\sigma)}}$ corresponding to a point $x \in (V_f^I)^{\langle (\underline{\lambda}, \sigma)\rangle}$ is the minimal positive $d$ such that $h^d x \in C_{G \rtimes S}(\underline{\lambda}, \sigma) x$. Since $(\underline{\lambda},\sigma)$ acts trivially on $(V_f^I)^{\langle (\underline{\lambda}, \sigma)\rangle}$, one has $C_{G \rtimes S}(\underline{\lambda}, \sigma) x = ({\rm Ker}\, A_\sigma \cap G)x$.
For $\underline{\nu} \in G$, one has
\begin{eqnarray*}
\underline{\nu} (V_f^I)^{\langle (\underline{\lambda}, \sigma)\rangle} & = & (V_f^I)^{\langle (\underline{\nu},1) (\underline{\lambda}, \sigma) (\underline{\nu}^{-1},1)\rangle} \\
&= & (V_f^I)^{\langle (A_\sigma(\underline{\nu})\underline{\lambda}, \sigma)\rangle} \, .
\end{eqnarray*}
If $\underline{\nu} \not\in {\rm Ker}\, A_\sigma \cap G$, then $A_\sigma(\underline{\nu})\underline{\lambda} \neq \underline{\lambda}$ and therefore $(V_f^I)^{\langle (A_\sigma(\underline{\nu})\underline{\lambda}, \sigma)\rangle} \cap  (V_f^I)^{\langle (\underline{\lambda}, \sigma)\rangle}  = \emptyset$. Therefore, for $x \in (V_f^I)^{\langle (\underline{\lambda}, \sigma)\rangle}$, $h^dx \in C_{G \rtimes S}(\underline{\lambda}, \sigma)$ if and only if $h^dx \in Gx$. This minimal $d$ is just $m_I$ computed in \cite{EG-PEMS}. This implies that
\[
\zeta_{h_{\vert (V_f^I)^{\langle (\underline{\lambda}, \sigma)\rangle}/C_{G \rtimes S}(\underline{\lambda},\sigma)}}(t) = (1-t^{m_I})^{r_I(\underline{\lambda},\sigma)} \, ,
\]
where
\[
r_I(\underline{\lambda},\sigma)= \frac{\chi((V_f^I)^{\langle (\underline{\lambda}, \sigma)\rangle})}{m_I|{\rm Ker}\, A_\sigma \cap G|} \, .
\]
According to \cite[Proposition~3.6]{EG-SIGMA}, $\chi((V_f^I)^{\langle (\underline{\lambda}, \sigma)\rangle})$ does not depend on $\underline{\lambda}$. Therefore
\[
\zeta_{h_{\vert (V_f^I)^{\langle (\underline{\lambda}, \sigma)\rangle}/C_{G \rtimes S}(\underline{\lambda},\sigma)}}(t) = \prod_{[\underline{\lambda}] \in G/({\rm Ker}\, A_\sigma \cap G)} (1- {\mathbf e}[-\age(\underline{\lambda})]t^{m_I})^{r_I} \, ,
\]
where $r_I= r_I(\underline{\lambda}, \sigma)$. The conjugacy classes of $(\underline{\lambda}, \sigma)$ are of the same size equal to $|A_\sigma(G)|$. The exponents ${\mathbf e}[-\age(\underline{\lambda})]$ are roots of degree $k_I$ with equal multiplicities (see \cite{EG-PEMS}). Therefore
\[
\zeta^{\orb,[\sigma]}_{h,V_f^I,G \rtimes S^I}(t) = (1-t^{\ell_I})^{R_I} \, ,
\]
where $\ell_I={\rm lcm}(m_I,k_I)$ (see \cite[Lemma~3.5]{EG-PEMS}).

In the same way,
\[
\zeta^{\orb,[\sigma]}_{\widetilde{h},V_{\widetilde{f}}^{\overline{I}},\widetilde{G} \rtimes S^{\overline{I}}}(t) = (1-t^{\widetilde{\ell}_{\overline{I}}})^{\widetilde{R}_{\overline{I}}} \, .
\]
In \cite{EG-PEMS}, it was shown that $\widetilde{\ell}_{\overline{I}}= \ell_I$. Moreover,
\[ 
\ell_I R_I =  \chi^{\orb, [\sigma]}(V_f^I, G\rtimes S^I), \quad \widetilde{\ell}_{\overline{I}} \widetilde{R}_{\overline{I}} =  \chi^{\orb, [\sigma]}(V_{\widetilde{f}}^{\overline{I}}, \widetilde{G}\rtimes S^{\overline{I}}) \, .
\]
By Theorem~\ref{theo:chi}, $\ell_I R_I =(-1)^n \widetilde{\ell}_{\overline{I}} \widetilde{R}_{\overline{I}}$. This proves the statement for proper $I$.

The proof of
\begin{equation}
\zeta^{\orb,[\sigma]}_{h,V_f^{I_0},G \rtimes S}(t) = \left( \zeta^{\orb,[\sigma]}_{{\rm id},{\rm pt},\widetilde{G} \rtimes S}(t) \right)^{(-1)^n}
  \label{eqn:orb_zeta0}
\end{equation}
follows the same scheme. The corresponding zeta functions are equal to $(1-t^{\ell_{I_0}})^{R_{I_0}}$ and $(1-t^{\ell_\emptyset})^{R_\emptyset}$ respectively. The exponent $\ell_{I_0}$ computed above (in this case $k_{I_0}=1$ and $\ell_{I_0}=m_{I_0}$) is equal to the exponent $\ell_\emptyset$ computed in \cite{EG-PEMS} (denoted by $\widetilde{k}$ there; one can say that $\widetilde{k}=k_\emptyset$). The fact that $\ell_{I_0}R_{I_0}=(-1)^n\ell_\emptyset R_\emptyset$ is again a consequence of Theorem~\ref{theo:chi}.
\end{proof*}

\section{Orbifold E-function}\label{sec:orb_E_function}
We shall now show that, under certain conditions, the conjecture for the orbifold E-function, namely Conjecture~\ref{conj:E}, holds for level 1. For this we need the mirror map \cite{EGT, Krawitz}. The arguments in \cite{EGT} are a little bit sketchy, here we give the precise definition.

\begin{sloppypar}

Let $f(x_1, \ldots , x_n)$ be an invertible polynomial and $G\subset G_f$. For $I \subset I_0:=\{ 1, \ldots , n \}$, let $\CC^I:=\{ (x_1, \ldots , x_n) \in \CC^n  : x_i =0 \mbox{ for } i \not\in I\}$ and let $\Omega^p(\CC^I)$ be the $\CC$-vector space of regular $p$-forms on $\CC^I$. For $\ulam \in G$, the restriction $f^{\ulam}$ of the polynomial $f$ to the fixed point set $\CC^{I_{\ulam}}$, where $I_{\ulam}$ is the set of the indices of coordinates fixed by $\ulam$, has an isolated critical point at the origin as well. Therefore the Milnor algebra  $A_{f^{\ulam}}:= \CC[x_i : i \in I_{\ulam}]/(\frac{\partial f}{\partial x_i}: i \in I_{\ulam})$ is finite dimensional. Recall that $n_{\ulam}=| I_{\ulam}|$. and let $\Omega_{f^{\ulam}}:= \Omega^{n_{\ulam}}(\CC^{I_{\ulam}})/(df^{\ulam} \wedge \Omega^{n_{\ulam}-1}(\CC^{I_{\ulam}}))$. This is a free $A_{f^{\ulam}}$-module of rank one.  We define a map $\psi: \bigoplus_{\ulam \in G}  \Omega_{f^{\ulam}} \to \CC[G_{\widetilde{f}}]$ on the natural basis 
of $\bigoplus_{\ulam \in G}  \Omega_{f^{\ulam}}$ as follows: Let $\wedge_{i \in I_{\ulam}} x_i^{k_i}dx_i \in \Omega_{f^{\ulam}}$. Define the element $(b_1, \ldots , b_n) \in \ZZ^n$ as follows. Let 
\[ b_i := \left\{ \begin{array}{cl} k_i+1 & \mbox{for } i \in I_{\ulam}, \\ 0 & \mbox{for } i \not\in I_{\ulam}. \end{array} \right.
\]
and let
\[
(\beta_1, \ldots , \beta_n):=(b_1, \ldots , b_n)E^{-1},
\] 
where $E$ is the matrix defining $f$. We set 
\[
\psi(\wedge_{i \in I_{\ulam}} x_i^{k_i}dx_i):= ({\mathbf e}[\beta_1], \ldots , {\mathbf e}[\beta_n]).
\]
From the definition of $G_{\widetilde{f}}$, it follows that $({\mathbf e}[\beta_1], \ldots , {\mathbf e}[\beta_n]) \in G_{\widetilde{f}}$.

Let $\ulam \in G$. The space $\CC^{I_{\ulam}}$ admits a natural $G$-action by restricting the $G$-action on $\CC^n$ to $\CC^{I_{\ulam}}$ (well-defined since $G$ acts diagonally on $\CC^n$). The form $\wedge_{i \in I_{\ulam}} x_i^{k_i}dx_i$ is $G$-invariant if and only if $\psi(\wedge_{i \in I_{\ulam}} x_i^{k_i}dx_i) \in\widetilde{G}$. Therefore the map $\psi$ restricts to a map $\bigoplus_{\ulam \in G}  \Omega^G_{f^{\ulam}} \to \CC[\widetilde{G}]$. This is the {\em mirror map}.  

Let $f = f_1 \oplus  \cdots \oplus f_s$ be the decomposition of $f$ into a Sebastiani-Thom sum of atomic polynomials (blocks) such that each $f_\nu$, $\nu=1, \ldots, s$ is either of Fermat, chain, or loop type. 
An element of the symmetry group $S$
can either permute blocks, act as the identity on a block or as a rotation on a block of loop type, see \cite[p.~12312]{EG-IMRN}.
 
Let $f_\nu$ be a loop type polynomial. It can be characterized by its length
$L$ and the sequence $a_1$, \dots, $a_{L}$ of the exponents.
That is
$$
f_\nu(x_1, \ldots, x_{L})= x_1^{a_1}x_2+x_2^{a_2}x_3+
\ldots, x_{L}^{a_{L}}x_1\,.
$$
The sequence $a_1$, \dots, $a_{L}$ of the exponents considered up to cyclic order is called the {\em type} of the loop.

\end{sloppypar}

\begin{theorem} \label{theo:E}
Let $f(x_1, \ldots , x_n)$ be an invertible polynomial with a symmetry group $G\rtimes S$ ($G \subset G_f$) such that
\begin{enumerate}
 \item[1)] in the Sebastiani-Thom decomposition $f = f_1 \oplus  \cdots \oplus f_s$ into atomic polynomials,
 there is not more than one loop of any type with the length $L$ even;
\item[2)]  For any subset $I \subset I_0:=\{ 1, \ldots , n \}$, the group $S^I$ consists of permutations $\sigma$ with $\sigma|_I$ being an even permutation of $I$.
\end{enumerate}
Then one has
\[
E^{[1]}(f,G \rtimes S)(t, \overline{t}) = (-1)^n E^{[1]}(\widetilde{f},\widetilde{G} \rtimes S)(t^{-1}, \overline{t}) \, .
\]
\end{theorem}

\begin{proof} 
Let $\ulam \in G$. By Assumption 2),
the group $S^{I_{\ulam}}$ acts on $\Omega^G_{f^{\ulam}}$ with determinant 1. Let  $C((\underline{\lambda},1)):=C_{G \rtimes S^{I_{\ulam}}}((\ulam,1))$ be the centralizer of the element $(\ulam,1)$ in $G \rtimes S^{I_{\ulam}}$. Hence the orbit of a monomial $\wedge_{i \in I_{\ulam}} x_i^{k_i}dx_i$  in 
$\Omega^{C((\underline{\lambda},1))}_{f^{\ulam}}$
with respect to the group $S^{I_{\ulam}}$  corresponds to the element $\sum_{s \in S^{I_{\ulam}}} \wedge_{i \in I_{\ulam}} x_{s^{-1}(i)}^{k_{s^{-1}(i)}}dx_{s^{-1}(i)}$ in $\Omega^{C((\ulam,1))}_{f^{\ulam}}$. 
Since $(1,s)^{-1}(\underline{\mu},1)(1,s)=(s^{-1}(\underline{\mu}),1)$ for $\underline{\mu} \in \widetilde{G}$ and $s \in S$, this orbit is mapped by the map $\psi$ to the conjugacy class of the element $(\psi(\wedge_{i \in I_{\ulam}} x_i^{k_i}dx_i ),1)$ in $\widetilde{G} \rtimes S$.

The group $G_f$ is the direct sum of the groups $G_{f_\nu}$, $\nu=1, \ldots , s$. Therefore an element $\underline{\lambda} \in G_f$ can be written as $\underline{\lambda}= (\underline{\lambda}_1, \ldots , \underline{\lambda}_s)$ with $\underline{\lambda}_\nu \in G_{f_\nu}$.
One can see that the map $\psi$ can be applied 
separately to each part of a monomial corresponding to the atomic
summands of $f$.
Two orbits of monomials can have the same image under the map $\psi$
only if their parts corresponding to a loop of even length
are mapped to $1$ (under the restriction of $\psi$ to this loop).
According to Assumption~1), the group $S$ preserves a
loop $f_\nu$ of even length. Therefore $S$ preserves the monomials in the variables
of $f_\nu$ whose image under the map $\psi$ is equal to $1$. 
The two elements of a loop of type $a_1, \dots, a_{L}$ which are mapped to 1 by $\psi$ are the elements $\underline{k}=(k_1, \ldots, k_L)$ with $k_{2i}=0$, $k_{2i-1}=a_{2i-1}-1$  and $k_{2i-1}=0$, $k_{2i}=a_{2i}-1$ for $i=1, \ldots, \frac{L}{2}$ (see the proof of \cite[Proposition~12]{EGT}). They are exchanged by a shift of an odd period $\ell$: $a_{j+\ell}=a_j$, $j=1, \ldots, L$, where $j+\ell$ is considered modulo $L$. The order of this shift is even. By Assumption~2), such a shift is excluded, see also Remark~\ref{rem:assumption2}.
Therefore $\psi$ is injective on an orbit of $S^{I_{\ulam}}$.

As in \cite{EGT}, define, for $\underline{\lambda} \in G$ and $\widetilde{\underline{\lambda}} \in \widetilde{G}$, the number $m_{\underline{\lambda},\widetilde{\underline{\lambda}}}$ as $m_{\underline{\lambda},\widetilde{\underline{\lambda}}} := 2^r$ where $r$ is the number of indices $\nu$ such that $f_\nu$ is a loop with an even number of variables and both components $\underline{\lambda}_\nu$ and $\widetilde{\underline{\lambda}}_\nu$ are equal to 1. 
Define
\[ \widehat{m}_{\underline{\lambda},\widetilde{\underline{\lambda}}} := \left\{ \begin{array}{cl} 0 & \mbox{if } \widetilde{\underline{\lambda}} \not\in \widetilde{G}_{(I_{\underline{\lambda}})} \, , \\ m_{\underline{\lambda},\widetilde{\underline{\lambda}}} & \mbox{otherwise}, \end{array} \right.
\]
where 
$\widetilde{G}_{(I_{\ulam})}$ is the set of elements of $\widetilde{G}$ which are images under the map $\psi$ of $\Omega_{f^{\ulam}}^G$.
Note that the number $\widehat{m}_{\underline{\lambda},\widetilde{\underline{\lambda}}}$ only depends on the conjugacy classes of the elements $(\underline{\lambda},1)$ and $(\widetilde{\underline{\lambda}},1)$.

Let $[(\underline{\lambda},1)]$ denote the conjugacy class of $(\underline{\lambda},1) \in G \rtimes S$. As in \cite{EGT}, one can derive an analogue of \cite[Proposition~14]{EGT}:
\begin{equation} \label{eqn:GGtilde}
E^{[1]}(f,G \rtimes S)(t, \overline{t}) = \sum \widehat{m}_{\underline{\lambda},\widetilde{\underline{\lambda}}}(-1)^{n_{\underline{\lambda}}}(t \overline{t})^{{\rm age}(\underline{\lambda}) - \frac{n-n_{\underline{\lambda}}}{2}} \left( \frac{\overline{t}}{t} \right)^{{\rm age}(\widetilde{\underline{\lambda}}) - \frac{n-n_{\widetilde{\underline{\lambda}}}}{2}} ,
\end{equation}
where the sum is over pairs of conjugacy classes $([(\underline{\lambda},1)],[(\widetilde{\underline{\lambda}},1)])$  in ${\rm Conj}(G \rtimes S) \times {\rm Conj}( \widetilde{G} \rtimes S)$. 
As in \cite{EGT}, one can see that this gives the formula of the theorem.
\end{proof}

\begin{remark}
The (somewhat strange) assumptions in Theorem~\ref{theo:E} are explained by the fact that they permit to make the proof close to the one
in~\cite{EGT} (or rather using the arguments from there), in particular to get Equation~(\ref{eqn:GGtilde})
(a version of Equation~(4.3) from \cite{EGT}).
Assumption~2) permits to identify collections of
monomials with invariant elements in the cohomology group.
Assumption~1) is needed to control the symmetry of
preimages of unit elements. Just for loops of even length the map $\psi$ is surjective, but the element 1 has two preimages.
If there is more than one loop of a type of this sort, there are several pairs of the preimages of 1
and it seems to be very difficult to control the action of $S$ on them: they may have different isotropy subgroups.
\end{remark}

\begin{remark} \label{rem:assumption2}
 Assumption~2) (under the condition that Assumption~1) holds) means the following:  in the Sebastiani-Thom decomposition $f = f_1 \oplus  \cdots \oplus f_s$ into atomic polynomials,
 \begin{enumerate}
  \item[a)] there is no loop of even length  $L$ such that an element of $S$ acts on it by a permutation of order $k$ with $L/k$ odd;
  \item[b)] for any loop of odd length, its orbit under the group $S$ consists of an odd number of its copies (including the loop itself);
  \item[c)] for any chain, its orbit under the group $S$ consists of an odd number of its copies (including the chain itself).
  \end{enumerate}
\end{remark}

By the remarks at the end of Section~\ref{sec:orb_notions}, Theorem~\ref{theo:E} implies the following corollary.
\begin{corollary} \label{cor:E}
Under the assumptions of Theorem~\ref{theo:E}, one has 
\[
  \overline{\zeta}^{\orb,[1]}_{f,G \rtimes S}(t)=
  \left(\overline{\zeta}^{\orb,[1]}_{\widetilde{f},\widetilde{G} \rtimes S}(t)\right)^{(-1)^n}\,.
\]
\end{corollary}

Combining Theorem~\ref{theo:orb_zeta} and Corollary~\ref{cor:E}, we obtain the following result.
\begin{theorem} 
Let $f(x_1, \ldots , x_n)$ be an invertible polynomial 
satisfying the conditions of Theorem~\ref{theo:E}. 
Let $G \subset G_f$ and let $S$ be a subgroup of $S_n$ satisfying PC and such that the centralizer $C_S(\sigma)$ of each element $\sigma \in S, \sigma \neq 1$,  is generated by $\sigma$. Then one has
\[
  \overline{\zeta}^{\orb}_{f,G \rtimes S}(t)=
  \left(\overline{\zeta}^{\orb}_{\widetilde{f},\widetilde{G} \rtimes S}(t)\right)^{(-1)^n}\,.
\]
\end{theorem}

\section{Example}\label{sec:example}
Here we compute the E-function for a rather simple example with $S=\ZZ_2=\langle s \rangle$. We showed that under some conditions at the level 1 the symmetry is essentially given by the mirror map defined by Krawitz. There exists the hope that on other levels the symmetry may be given by certain analogues of  the Krawitz mirror map. It appears to be difficult (if possible) to find this analogue even for this simple example.

Let 
\[ f(x_1,x_2,x_3,x_4)=x_1^3x_2+x_2^5x_3+x_3^3x_4+x_4^5x_1
\]
(of loop type), $S=\ZZ_2=\langle (13)(24) \rangle$, $G_f=\ZZ_{224}=\langle \sigma \rangle$, where $\sigma$ can be taken of the form 
$\sigma = \frac{1}{224}(1,-3,15,-45)$.
Here we use the notation $(\alpha_1, \alpha_2, \alpha_3, \alpha_4)$ ($\alpha_i \in \QQ$) for the element 
$({\mathbf e}[\alpha_1], {\mathbf e}[\alpha_2], {\mathbf e}[\alpha_3], {\mathbf e}[\alpha_4]) \in (\CC^\ast)^4$.
Note that $\widetilde{f}=f$. Any subgroup of $G_f$ is preserved by $S$. Let us take $G=\ZZ_8 = \langle \delta \rangle$, where
$\delta = \sigma^{28} = \frac{1}{8}(1,-3,-1,3)$
(and therefore $\widetilde{G}=\ZZ_{28}$).

At the level 1, one has
\begin{eqnarray*} \lefteqn{E^{[1]}(f,G \rtimes S)(t, \overline{t}) = E^{[1]}(\widetilde{f},\widetilde{G} \rtimes S)(t^{-1}, \overline{t})} \\
&= & \left(\frac{\overline{t}}{t} \right)^{-\frac{8}{7}} +\left(\frac{\overline{t}}{t} \right)^{-\frac{6}{7}} + 2\left(\frac{\overline{t}}{t} \right)^{-\frac{4}{7}}+ \left(\frac{\overline{t}}{t} \right)^{-\frac{3}{7}}+3\left(\frac{\overline{t}}{t} \right)^{-\frac{2}{7}}+\left(\frac{\overline{t}}{t} \right)^{-\frac{1}{7}}+8 \\
& & {} + \left(\frac{\overline{t}}{t} \right)^{\frac{1}{7}}+3\left(\frac{\overline{t}}{t} \right)^{\frac{2}{7}}+\left(\frac{\overline{t}}{t} \right)^{\frac{3}{7}}+2\left(\frac{\overline{t}}{t} \right)^{\frac{4}{7}}+\left(\frac{\overline{t}}{t} \right)^{\frac{6}{7}}+\left(\frac{\overline{t}}{t} \right)^{\frac{8}{7}}.
\end{eqnarray*}

At the level $[s]$, we have two conjugacy classes of elements $(\lambda , s) \in G \rtimes S$, namely, $[(1,s)]$ and $[(\delta,s)]$. For the representatives $(1,s)$ and $(\delta, s)$ of them,  the fixed point sets are 2-dimensional and are given by the equations $x_3=ax_1$, $x_4=bx_2$ for certain complex numbers $a,b$, where $a=b=1$ for $(1,s)$. Therefore, in the coordinates $x_1$ and $x_2$ the Milnor fibres inside them are given by the equation $a'x_1^3x_2+b'x_2^5x_1=1$ ($a'=b'=2$ for $(1,s)$). A basis of $\Omega_{f^{\langle ( \bullet , s) \rangle}}$ is given by the monomials $x_1^{k_1}x_2^{k_2}dx_1dx_2$ with $0 \leq k_1 \leq 2$, $0 \leq k_2 \leq 4$. The centralizers of $(1,s)$ and of $(\delta , s)$ are $\langle (\delta^4,1),(1,s) \rangle$ and $\langle (\delta^4,1), (\delta,s) \rangle$ respectively, where $\delta^4= \frac{1}{2}(1,1,1,1)$. Therefore the monomials which give a basis of $\Omega_{f^{\langle ( \bullet , s) \rangle}}^{C_{G \rtimes S}(\bullet,s)}$ are those of even order. The ages of the elements $(1,s)$ and $(\delta,s)$ are equal to 1. Therefore, in the E-function on level $[s]$, we only have to count the weights of the monomials in Table~\ref{tab:mirror}.
\begin{table}[h]
\begin{center}
\begin{tabular}{ccc}
\hline
Monomials & Weights & Elements of $\widetilde{G} \rtimes S$ \\
\hline
$dx_1dx_2$ & $- \frac{4}{7}$ & $\left\{ \begin{array}{c} (\frac{1}{28}(18,2,18,2),s) \sim (\frac{1}{28}(4,16,4,16),s), \\ (\frac{1}{28}(25,23,11,9),s) \sim (\frac{1}{28}(11,9,25,23),s) \end{array} \right\}$ \\
$x_1^2dx_1dx_2$ & $0$ &  \\
$x_1x_2dx_1dx_2$ & $ - \frac{1}{7}$ & $\left\{ \begin{array}{c} (\frac{1}{28}(8,4,8,4),s) \sim (\frac{1}{28}(22,18,22,18),s), \\ (\frac{1}{28}(15,25,1,11),s) \sim (\frac{1}{28}(1,11,15,25),s) \end{array} \right\}$ \\
$x_2^2dx_1dx_2$ & $- \frac{2}{7}$ & $\left\{ \begin{array}{c} (\frac{1}{28}(16,8,16,8),s) \sim (\frac{1}{28}(2,22,2,22),s), \\ (\frac{1}{28}(9,15,23,1),s) \sim (\frac{1}{28}(23,1,9,15),s) \end{array} \right\}$ \\
$x_1^2x_2^2dx_1dx_2$ & $\frac{2}{7}$ & $\left\{ \begin{array}{c} (\frac{1}{28}(26,6,26,6),s) \sim (\frac{1}{28}(12,20,12,20),s) , \\ (\frac{1}{28}(5,27,19,13),s) \sim (\frac{1}{28}(19,13,5,27),s) \end{array} \right\}$ \\
$x_1x_2^3dx_1dx_2$ & $\frac{1}{7}$ & $\left\{ \begin{array}{c} (\frac{1}{28}(6,10,6,10),s) \sim (\frac{1}{28}(20,24,20,24),s), \\ (\frac{1}{28}(27,17,13,3),s) \sim (\frac{1}{28}(13,3,27,17),s) \end{array} \right\}$ \\
$x_2^4dx_1dx_2$ & $0$ & \\
$x_1^2x_2^4dx_1dx_2$ & $\frac{4}{7}$ & $\left\{ \begin{array}{c} (\frac{1}{28}(24,12,24,12),s) \sim (\frac{1}{28}(10,26,10,26),s), \\  (\frac{1}{28}(17,19,3,5),s) \sim (\frac{1}{28}(3,5,17,19),s) \end{array} \right\}$ \\
\hline
\end{tabular}
\end{center}
\caption{Correspondence between monomials and elements of $\widetilde{G} \rtimes S$} \label{tab:mirror}
\end{table}

The weight of an element $(\widetilde{g},s) \in \widetilde{G} \rtimes S$ is equal to ${\rm age}(\widetilde{g})- \frac{n-n_{\widetilde{g}}}{2}$. 
For the weight 0 there are two conjugacy classes $\{(\widetilde{\delta}^{14},s), (\widetilde{\delta}^0,s)\}$ and $\{ (\widetilde{\delta}^7,s),(\widetilde{\delta}^{21},s)\}$, where 
\[ \widetilde{\delta}^{14}=\frac{1}{2}(1,1,1,1),  \quad \widetilde{\delta}^7=\frac{1}{4}(1,-1,-1,1), \quad \widetilde{\delta}^{21}= \frac{1}{4}(-1,1,1,-1).
\]
Representatives of them have 2-dimensional fixed point sets and they give 4 generators with weight $- \frac{n_g}{2}=0$ (two of them each).

From Table~\ref{tab:mirror} one gets
\begin{eqnarray*} \lefteqn{E^{[s]}(f,G \rtimes S)(t, \overline{t}) = E^{[s]}(\widetilde{f},\widetilde{G} \rtimes S)(t^{-1}, \overline{t})} \\
&= & 2 \left( \left(\frac{\overline{t}}{t} \right)^{-\frac{4}{7}}+ \left(\frac{\overline{t}}{t} \right)^{-\frac{2}{7}}+\left(\frac{\overline{t}}{t} \right)^{-\frac{1}{7}}+2  + \left(\frac{\overline{t}}{t} \right)^{\frac{1}{7}}+\left(\frac{\overline{t}}{t} \right)^{\frac{2}{7}}+\left(\frac{\overline{t}}{t} \right)^{\frac{4}{7}} \right).
\end{eqnarray*}

Looking at the table, one sees that the mirror map (if it exists) should send, e.g., the monomial $dx_1dx_2$ originated from one of the two conjugacy classes to the conjugacy class either of $(\frac{1}{28}(18,2,18,2),s)$ or of $(\frac{1}{28}(25,23,11,9),s)$. We did not succeed to find an analogue of the mirror map for this case.

\section*{Acknowledgements}
This work has been partially supported by DFG. The work of the second author
(Sections~\ref{sec:orb_notions}, \ref{sec:among_levels}, \ref{sec:orb_zeta}, \ref{sec:example}) was
supported by the grant 21-11-00080 of the Russian Science Foundation.
We would like to thank the referee for carefully reading our paper and for their useful comments which helped to improve the paper.


\bigskip
\noindent Leibniz Universit\"{a}t Hannover, Institut f\"{u}r Algebraische Geometrie,\\
Postfach 6009, D-30060 Hannover, Germany \\
E-mail: ebeling@math.uni-hannover.de\\

\medskip
\noindent Moscow State University, Faculty of Mechanics and Mathematics,\\
Moscow Center for Fundamental and Applied Mathematics,\\
Moscow, GSP-1, 119991, Russia\\
\& National Research University ``Higher School of Economics'',\\
Usacheva street 6, Moscow, 119048, Russia.\\
E-mail: sabir@mccme.ru

\end{document}